\documentclass{amsart}
\usepackage{bbm}
\usepackage{amsxtra,amscd}
\usepackage{graphicx}
\usepackage{amsmath}
\usepackage{amsfonts}
\usepackage{amssymb}

\newtheorem{theorem}{Theorem}[section]
\newtheorem{lemma}[theorem]{Lemma}
\theoremstyle{definition}

\newtheorem{example}[theorem]{Example}

\newtheorem{proposition}[theorem]{Proposition}
\newtheorem{remark}[theorem]{Remark}
\newtheorem{claim}[theorem]{Claim}

\numberwithin{equation}{section}

\begin{document}

\title{On the Transcendental Galois Extensions}

\author{Feng-Wen An}
\address{School of Mathematics and Statistics, Wuhan University, Wuhan,
Hubei 430072, People's Republic of China}
\email{fwan@amss.ac.cn}



\begin{abstract}
In this paper the transcendental Galois extensions of a field will be introduced as counterparts to algebraic Galois ones. There exist several types of transcendental Galois extensions of a given field, from the weakest one to the strongest one, such as Galois, tame Galois, strong Galois, and absolute Galois. The four Galois extensions are distinct from each other in general but
coincide with each other for cases of algebraic extensions. The transcendental Galois extensions arise from higher relative dimensional Galois covers of arithmetic schemes.
In the paper we will obtain several properties of Galois extensions in virtue of conjugation and quasi-galois and will draw a comparison between the Galois extensions. Strong Galois is more accessible. It will be proved that a purely transcendental extension is strong Galois.
\end{abstract}

\maketitle

\section*{Introduction}

In this paper we will introduce the transcendental Galois extensions of fields. Let $L$ be an arbitrary extension of a field $K$ (not necessarily algebraic). We will have four types of extensions of fields: \emph{Galois}; \emph{tame Galois}; \emph{strong Galois}; \emph{absolute Galois}. These Galois extensions,
from the weakest one to the strongest one, are distinct from each other for transcendental extensions in general.

Naturally, we have the first definition that $L$ is  Galois over $K$ if $K$ is the invariant subfield of the Galois group $Gal(L/K)$. In general, such a transcendental Galois extension $L/K$ is more complicated than an algebraic one; the Galois group $Gal(L/K)$ is always an infinite group, and $Gal(L/K)$ is not a profinite group. So many approaches to algebraic extensions, established upon finite dimensions,  will not be valid. In particular,
such a transcendental Galois extension $L/K$ will not behave like an algebraic one; many properties of algebraic Galois extensions will never hold.

It leads us to take the other types of Galois extensions. In deed, a transcendental absolute Galois extension behaves exactly like an algebraic one. At the same time, strong Galois is weaker than absolute Galois but stronger than Galois and than tame Galois. For the cases of algebraic extensions, the four Galois will coincide with each other.

The transcendental Galois extensions are encountered by us in the discussions on function fields of higher relative dimensional Galois covers of arithmetic schemes (for instances, see \cite{an1,an2}).  At least to the author's knowledge, we have not found a successful discussion on transcendental Galois extensions of fields.

As counterparts to \emph{conjugate} and \emph{normal} in algebraic extensions of fields, we will introduce notions of \emph{conjugations} and \emph{quasi-galois} for transcendental extensions of fields in \S 1, respectively. For algebraic extensions, it is a preliminary fact that
$$\emph{Galois} =\emph{a unique conjugation} + \emph{separably generated}.$$

Naturally, in \S 3 we will have several properties of Galois extensions in virtue of conjugation and quasi-galois and then we will draw a comparison between the Galois extensions. In deed, for transcendental extensions of fields, we have
$$\emph{strong Galois} =\emph{a unique strong conjugation} + \emph{separably generated};$$
$$\emph{absolute Galois} =\emph{a unique absolute conjugation} + \emph{separably generated}.$$
However,
$$\emph{tame Galois} \not=\emph{a unique conjugation} + \emph{separably generated}.$$

This also shows that transcendental Galois extensions of fields are much more complicated than algebraic ones.

At last we will prove that a purely transcendental extension is strong Galois. Then it will be seen that that a strong Galois extension can be regarded as a purely transcendental extension of an algebraic Galois extension of a field with respect to all possible linearly disjoint bases. In deed, strong Galois is more accessible than tame Galois.

Here, the approach to transcendental Galois extensions discussed in the paper is based on preliminary calculus of polynomials over fields. In \S 2 a  preparatory lemma is obtained from Weil's theory of specializations (see \cite{weil}).

As a conclusion, we give some notes on a transcendental Galois extension $L/K$. The main feature of such a Galois extension is that for an intermediate subfield $K\subseteq F\subseteq L$, in general, $L$ is not a Galois extension of $F$. In the terminology of the present paper, that is to say that a Galois extension is not necessarily  absolute Galois and that there exist many other types of Galois extensions such as strong Galois.
It is natural for one to take algebraic operations on transcendental Galois extensions  and to consider the Galois correspondences between the subfields and the subgroups.  However, even for strong Galois, the case is  complicated very much. For instance, consider the transcendental extension $\overline{\mathbb{Q}(s,t)}/\mathbb{Q}$ in two variables.

\section{Definitions for Transcendental Galois Extensions}

Let $L$ be an extension of a field $K$ (not
necessarily algebraic). Let $\overline{L}$ denote an fixed algebraic closure of $L$.

\subsection{Notations}

The \textbf{Galois group} of $L$ over $K$ is the group of all automorphisms $\sigma$ of $L$ such that $\sigma (x)=x$ holds for any $x\in K$.

By a \textbf{nice basis} $(\Delta, A)$ of $L$ over $K$, we understand that $\Delta$ is a transcendental basis of $L$ over $K$ and $A$ is a linearly basis of $L$ over $K(\Delta)$ such that $L=K(\Delta)[A]$ is algebraic over $K(\Delta)$; moreover, such a nice basis $(\Delta, A)$ is  a \textbf{linearly disjoint basis} of $L$ over $K$ if the fields $K(\Delta)$ and $K(A)$ are linearly disjoint over $K$. Put
\begin{itemize}
\item $\Sigma[L/K]\triangleq$ the set of nice bases of $L$ over $K$;

\item $\Sigma[L/K]_{ld}\triangleq$ the set of linearly disjoint basis of $L$ over $K$.
\end{itemize}

\subsection{Several types of conjugations}

Let $M$ be a field with $K\subseteq L\bigcap M$.

In general, there exist several types of conjugations  which behave like conjugates of a field in an algebraic extension.
\begin{itemize}
\item $M$
 is said to be a \textbf{conjugation} of $L$ over $K$ if there is some $(\Delta, A)\in \Sigma[L/K]$
 such that
 $M$ is a conjugate of $L$ over the subfield $K(\Delta)$.

 In such a case, $M$ is also said to be a \textbf{conjugation} of $L$ \textbf{with respect to the nice basis} $(\Delta, A)$.

\item $M$
 is said to be a \textbf{strong conjugation} of $L$ over $K$ if the two properties are satisfied:
     \begin{itemize}

     \item $\Sigma[L/K]_{ld}$ is a nonempty set;

     \item For each $(\Delta, A)\in \Sigma[L/K]_{ld}$,
 $M$ is a conjugate of $L$ over the subfield $K(\Delta)$.
     \end{itemize}
\end{itemize}

It is clear that a conjugation of $L$ over $K$ must be contained in $\overline{L}$ as a subfield.

\subsection{Several types of quasi-galois extensions}

Also there are several types of quasi-galois extensions that behave like normal algebraic extensions.
\begin{itemize}

\item  $L$ is said to be \textbf{quasi-galois} over $K$ if there is some $(\Delta, A)\in \Sigma[L/K]$ such that each irreducible
polynomial $f(X)$ over the subfield $K(\Delta)$ that has a root in $L$ factors completely in $L\left[ X\right] $ into linear factors.

In such a case, $L$ is also said to be \textbf{quasi-galois} over $K$ \textbf{with respect to} $(\Delta, A)$.

\item  $L$ is said to be \textbf{strong quasi-galois} over $K$ if the two conditions are satisfied:
     \begin{itemize}
     \item $\Sigma[L/K]_{ld}$ is a nonempty set;

     \item For any $(\Delta, A)\in \Sigma[L/K]_{ld}$, each irreducible
polynomial $f(X)$ over the subfield $K(\Delta)$ that has a root in $L$ factors completely in $L\left[ X\right] $ into linear factors.
     \end{itemize}

\item  $L$ is said to be \textbf{absolute quasi-galois} over $K$ if for any subfield $F$ with $K\subseteq
F\subseteq L$, each irreducible
polynomial $f(X)$ over $F$ that has a root in $L$ factors completely in $L\left[ X\right] $ into linear factors.
\end{itemize}

\begin{remark}
For an arbitrary extension of a field, we have
\begin{equation*}
\begin{array}{l}
\text{absolute quasi-galois (with linear disjoint bases)} \\

\implies  \text{strong quasi-galois} \\

 \implies  \text{quasi-galois}.
\end{array}
\end{equation*}
In general, the converse is not true. However, for algebraic extensions, the four Galois are the same.
\end{remark}

\subsection{Several types of Galois extensions}

A transcendental Galois extension is more complicated than an algebraic one.

In deed, we have several types of Galois extensions such as the following.
\begin{itemize}

\item $L$ is said to be \textbf{Galois} over $K$  if $K$ is the
invariant subfield of $Gal(L/K)$, i.e., if there is an identity $$K=\{x\in L :\sigma (x)=x
\text{ for any }\sigma \in Gal(L/K) \}.$$

\item  $L$ is said to be \textbf{tame Galois} over $K$  if there is some $(\Delta, A)\in \Sigma[L/K]$ such that $L$ is Galois over the subfield $K(\Delta)$.

\item  $L$ is said to be \textbf{strong Galois} over $K$  if the two conditions are satisfied:
     \begin{itemize}
     \item $\Sigma[L/K]_{ld}$ is a nonempty set;

     \item $L$ is an algebraic Galois extension of $K(\Delta)$ for any $(\Delta, A)\in\Sigma[L/K]_{ld}$.
     \end{itemize}

\item  $L$ is said to be \textbf{absolute Galois} over $K$  if $L$ is Galois over any subfield $F$ such that $K\subseteq F\subseteq L$.
\end{itemize}

\begin{remark}
For an arbitrary extension of a field, we have
\begin{equation*}
\begin{array}{l}
\text{absolute Galois (with linear disjoint bases)} \\

\implies  \text{strong Galois} \\

 \implies  \text{tame Galois} \\

  \implies \text{Galois}
\end{array}
\end{equation*}
by virtue of properties obtained in \S 3.
In general, the converse is not true. However, for algebraic extensions, the four Galois are the same.
\end{remark}

\begin{remark}
In general, there exist many other types of Galois extensions from Galois extensions to absolute Galois extensions. Here, Galois and tame Galois arise from infinite Galois covers of arithmetic schemes (\cite{an1,an2}).
\end{remark}

\subsection{Examples}

As a conclusion to this section, we take several examples.

\begin{example}
Let $t$ be a variable over $\mathbb{Q}$ and let $g(x)=x^{3}+x$ be a polynomial over $\mathbb{Q}$.  Then \begin{itemize}
\item $\mathbb{Q}(t)/\mathbb{Q}$ is Galois;

\item $\mathbb{Q}(t)/\mathbb{Q}$ is not absolute Galois
\end{itemize}
since $\mathbb{Q}(t)$ is not Galois over $\mathbb{Q}(g(t))$ (see \cite{m-b}).
\end{example}

\begin{example}
Let $s$ and $t$ be two variables over  $\mathbb{Q}$.  By \emph{Propositions 3.9-10} in \S 3 it is seen that
\begin{itemize}
\item $\mathbb{Q}(\sqrt{2},t,t^{\frac{1}{2}})/\mathbb{Q}$ is tame Galois but is not strong Galois;

\item $\mathbb{Q}(\sqrt{2},s,t)/\mathbb{Q}$ is strong Galois  but is not absolute Galois.
\end{itemize}
\end{example}

\begin{example}
Let $t$ be a variable over the complex number field $\mathbb{C}$. Take a subfield $\mathbb{C}\subsetneqq L\subsetneqq \mathbb{C}(t)$. There are the following statements:
\begin{itemize}
\item $\mathbb{C}(t)/\mathbb{C}$ is absolute Galois.

\item $\mathbb{C}(t)/L$ is absolute Galois.

\item Let $L/\mathbb{C}$ be Galois. Then $Gal(L/\mathbb{C})$ is an infinite group.

\item $L$ is not isomorphic to $\mathbb{C}(t)$ over $\mathbb{C}$. In particular, for any $n \geq 2$, the subfield $\mathbb{C}(t^{n})$  is not isomorphic to $\mathbb{C}(t)$ over $\mathbb{C}$
 from \emph{Proposition 3.6} in \S 3 according to the fact that $\mathbb{C}(t)$ has a unique conjugation over $\mathbb{C}$.
\end{itemize}
\end{example}

\section{Preliminaries}

Let $L$ be a finitely generated extension of a field $K$.
 Recall that a \textbf{$(r,n)$-nice basis} of $L$ over $K$ we understand a finite number of elements $w_{1},w_{2},\cdots ,w_{n}\in L $ satisfying the below conditions:
\begin{itemize}
\item $L=K(w_{1},w_{2},\cdots ,w_{n})$;

\item $w_{1},w_{2},\cdots ,w_{r}$ are a transcendental basis of $L$ over $K$;

\item $w_{r+1},w_{r+2},\cdots ,w_{n}$ are a linear basis of $L$ over $K(w_{1},w_{2},\cdots ,w_{r})$,
\end{itemize}
where $0\leq r\leq n$.

\begin{lemma}
Let $F$ be a subfield with $K\subseteq F
\subsetneqq L$. Suppose that  $x \in L$ is algebraic over $F$ and that $z\in \overline{L}$ is a
conjugate of $x$ over $F$.

Then there exists a $(s,m)$-nice basis
$v_{1},v_{2},\cdots ,v_{m}$ of $L$ over $F\left( x\right)$ and an $F$-isomorphism $\tau$ from the field
\begin{equation*}
F\left( x,v_{1},v_{2},\cdots ,v_{s},v_{s+1},\cdots, v_{m}\right)=L
\end{equation*}
onto a field of the form
\begin{equation*}
F\left( z,v_{1},v_{2},\cdots ,v_{s},w_{s+1},\cdots, w_{m}\right)=\tau(L)
\end{equation*}
such that
\begin{equation*}
\tau (x)=z,\tau (v_{1})=v_{1},\cdots,\tau (v_{s})=v_{s},
\end{equation*}
where $w_{s+1},w_{s+2},\cdots, w_{m}$ are elements contained in an extension of $F$.

Moreover, we have
\begin{equation*}
w_{s+1}=v_{s+1},w_{s+2}=v_{s+2},\cdots, w_{m}=v_{m}
\end{equation*}
if $z$ is not contained in $F(v_{1},v_{2},\cdots,v_{m})$.
\end{lemma}

\begin{proof}
We will proceed in two steps according to the assumption that $s=0$
or $s\not= 0$.

\emph{Step 1}. Let $s\not= 0$.
That is, $v_{1}$ is a variable over $F\left( x\right) $.

Let $\sigma_{x}$ be the $F-$isomorphism between fields $F(x)$ and
$F(z)$ with $\sigma_{x}(x)=z$. From the isomorphism $\sigma_{x}$ we
obtain an isomorphism $\sigma _{1}$ of $F\left( x,v_{1}\right) $
onto $F\left( z,v_{1}\right) $ defined by
$$
\sigma_{1}:\frac{f(v_{1})}{g(v_{1})}\mapsto \frac{\sigma _{x}(f)(v_{1})}{%
\sigma _{x}(g)(v_{1})}
$$
for any polynomials $$ f[X_{1}],g[X_{1}]\in F\left( x\right) [X_{1}]
$$ with $$g[X_{1}]\neq 0.$$

It is easily seen that $$g(v_{1})=0$$ if and only if $${\sigma _{x}(g)(v_{1})}
=0 .$$ Hence, the map $\sigma_{1}$ is well-defined.

Similarly, for the elements $$v_{1},v_{2},\cdots ,v_{s}\in L$$ that
are variables over $F(x)$, there is an isomorphism
$$
\sigma _{s}: F\left( x,v_{1},v_{2},\cdots ,v_{s}\right)\longrightarrow
F\left( z,v_{1},v_{2},\cdots ,v_{s}\right)
$$
of fields defined by
$$
x\longmapsto z\text{ and }v_{i}\longmapsto v_{i}
$$
for $1\leq i\leq s$, where we have the restrictions
$$
\sigma_{i+1}|_{F\left( x,v_{1},v_{2},\cdots ,v_{i}\right)}=\sigma_{i}
.$$

If $s=m$, we have $$L=F\left( x,v_{1},v_{2},\cdots ,v_{s}\right).$$
Then the field $F\left( z,v_{1},v_{2},\cdots
,v_{s}\right)$ is an $F-$conjugation of $L$.

In the following we put $$s\leqslant
m-1.$$

\emph{Step 2}. Let $s=0$. It reduces to consider the case
$v_{s+1}\in L$ since $v_{s+1}$ is algebraic over the field
$F(v_{1},v_{2},\cdots ,v_{s})\subseteq L$.

There are two cases for the
element $v_{s+1}$.

\emph{Case (i)}. Suppose that $z$ is not contained in
$F(v_{1},v_{2},\cdots,v_{s+1})$.

There is an isomorphism $$\sigma _{s+1}:F\left(
x,v_{1},v_{2},\cdots ,v_{s+1}\right) \cong F\left(
z,v_{1},v_{2},\cdots ,v_{s+1}\right) $$ of fields given by
\begin{equation*}
x\longmapsto z\text{ and }v_{i}\longmapsto v_{i}
\end{equation*}
with $1\leq i\leq s+1.$

It is seen that the map $\sigma_{s+1}$ is well-defined. In deed, by the below
\emph{Claim 2.2} it is seen that $$f\left( v_{s+1}\right) =0$$
holds if and only if $$\sigma_{s}\left( f\right) \left(
v_{s+1}\right) =0$$ holds for any polynomial $$f\left( X_{s+1}\right)
\in F\left( x,v_{1},v_{2},\cdots,v_{s}\right) \left[ X_{s+1}\right].
$$

\emph{Case (ii)}. Suppose that $z$ is contained in the field
$F(v_{1},v_{2},\cdots,v_{s+1})$.

By the below \emph{Claim 2.3} we have an  element
$v_{s+1}^{\prime}$ contained in an extension of $F$ such that the
fields $F(x,v_{s+1})$ and $F(z, v_{s+1}^{\prime})$ are isomorphic
over $F$.

Then by the same procedure as in \emph{Case (i)} of
\emph{Claim 2.2} it is seen that the two fields
$$F(x,v_{s+1},v_{1},v_{2},\cdots,v_{s})\cong F(z,v_{s+1}^{\prime},v_{1},v_{2},\cdots,v_{s})$$ are isomorphic over
$F$.

Hence, in such a manner we have an $F-$isomorphism $ \tau$ from the
field
$$F\left( x,v_{1},v_{2},\cdots ,v_{s},v_{s+1},\cdots, v_{m}\right)$$
onto the field of the form $$F\left( z,v_{1},v_{2},\cdots
,v_{s},w_{s+1},\cdots, w_{m}\right)$$ such that $$\tau (x)=z,\tau
(v_{1})=v_{1},\cdots,\tau (v_{s})=v_{s},$$ where
$$w_{s+1},w_{s+2},\cdots, w_{m}$$ are elements contained in an
extension of $F$. This completes the proof.
\end{proof}

\begin{claim}
Take an element $f$ in the polynomial ring $F\left[
X,X_{1},X_{2},\cdots,X_{s+1}\right]$. Suppose that $z$ is not
contained in the field $F(v_{1},v_{2},\cdots,v_{s+1})$. Then
$$f\left( x,v_{1},v_{2},\cdots,v_{s+1}\right) =0$$ holds if and only
if $$f\left( z,v_{1},v_{2},\cdots,v_{s+1}\right) =0$$ holds.
\end{claim}

\begin{proof}
Here we use Weil's algebraic theory of specializations (See \cite{weil})
 to give the proof. For the element $v_{s+1}$, there is either
$$v_{s+1}\in \overline{F}$$  or
$$v_{s+1}\in \overline{F(v_{1},v_{2},\cdots,v_{s+1})}\setminus
\overline{F}$$ where $\overline{F}$ denotes the algebraic closure
of the field $F$.

\emph{Case (i)}. Let $v_{s+1}\in
\overline{F(v_{1},v_{2},\cdots,v_{s+1})}\setminus \overline{F}$.

By \emph{Theorem 1} in \cite{weil}, \emph{Page 28},  it is clear that
$\left( z\right) $ is a (generic) specialization of $\left( x\right)
$ over $F$ since $z$ and $x$ are conjugates over
 $F$.
From \emph{Proposition 1} in \cite{weil}, \emph{Page 3}, it is seen that
$F\left( v_{1},v_{2},\cdots ,v_{s+1}\right) $ and the field $F(x)$
are free with respect to each other over $F$ since $x$ is algebraic
over $F$. Hence, $F\left( v_{1},v_{2},\cdots ,v_{s+1}\right) $ is
a free field over $F$ with respect to $(x)$.

By \emph{Proposition 3}
in \cite{weil}, \emph{Page 4}, it is seen that $F\left( v_{1},v_{2},\cdots
,v_{s+1}\right) $ and the algebraic closure $\overline{F}$ are
linearly disjoint over $F$. Then $F\left( v_{1},v_{2},\cdots
,v_{s+1}\right) $ is a regular extension of $F$ (see
\cite{weil}, \emph{Page 18}).

In virtue of \emph{Theorem 5} in \cite{weil}, \emph{Page 29}, it is seen that
$$\left( z,v_{1},v_{2},\cdots,v_{s+1}\right) $$ is a (generic)
specialization of $$\left( x,v_{1},v_{2},\cdots,v_{s+1}\right) $$ over
$F$ since there are two specializations:

\qquad $\left( z\right) $ is a (generic) specialization of
$\left( x\right) $ over $F$;

\qquad $\left(v_{1},v_{2},\cdots,v_{s+1}\right) $ is a (generic)
specialization of $\left( v_{1},v_{2},\cdots,v_{s+1}\right) $
over $F$.

\emph{Case (ii)}. Let $v_{s+1}\in \overline{F}$.

By the assumption for $z$ it is seen that $z$ is not contained
in the field $F(v_{s+1})$. Then there is an
isomorphism between the fields $F(x,v_{s+1})$ and $F(z,v_{s+1})$. It
follows that $(z,v_{s+1})$ is a (generic) specialization of
$(x,v_{s+1})$ over $F$.

From the same procedure as in the above
\emph{Case (i)} it is seen that $$\left(
z,v_{s+1},v_{1},v_{2},\cdots,v_{s}\right) $$ is a (generic)
specialization of $$\left( x,v_{s+1},v_{1},v_{2},\cdots,v_{s}\right)
$$ over $F$.

Now take any polynomial $f\left( X,X_{1},X_{2}
,\cdots,X_{s+1}\right)$ over $F$. According to \emph{Cases
(i)-(ii)}, it is seen that $$f\left(
x,v_{1},v_{2},\cdots,v_{s+1}\right) =0$$ holds if and only if
$$f\left( z,v_{1},v_{2},\cdots,v_{s+1}\right) =0$$ holds from the preliminary facts of
 generic specializations.
\end{proof}

\begin{claim}
Assume that $F(u)$ and $F(u^{\prime})$
are isomorphic over $F$ given by $u\mapsto u^{\prime}$. Let $w$ be
an element contained in an extension of $F$. Then there is an
element $w^{\prime}$ contained in some extension of $F$ such that
the fields $F(u,w)$ and $F(u^{\prime},w^{\prime})$ are isomorphic
over $F$.
\end{claim}

\begin{proof}
It is immediate from \emph{Proposition 4} in \cite{weil}, \emph{Page 30}.
\end{proof}

\section{Properties of Transcendental Galois Extensions}

In this section we will have several properties of Galois extensions in virtue of conjugation and quasi-galois and then we will draw a comparison between the Galois extensions.

\subsection{Quasi-galois and conjugation}

Let $L$ be a finitely generated extension of a field $K$.

\begin{lemma}
The following statements are equivalent:
\begin{itemize}
\item $L$ is quasi-galois over $K$.

\item Each conjugation of $L$ with respect to some nice basis of $L$ over $K$ must be contained in $L$.

\item There is a nice basis $(\Delta, A)\in \Sigma[L/K]$ such that $L$
contains each conjugate of $F\left( x\right) $ over $F$ for any $x\in L\setminus K(\Delta)$ and any subfield $K(\Delta)\subseteq F\subseteq L$.
\end{itemize}
\end{lemma}

\begin{proof}
$1st\Longleftrightarrow 3rd$. It is immediate from the fact that  $F(x)$ is algebraic over $F$.

$2ed\implies 3rd$. Take an $x\in L\setminus K(\Delta)$ and a subfield $K(\Delta)\subseteq F\subseteq L$.
Then $x$ is algebraic over $F$.

Let $z$ be an $F$-conjugate of $x $.
If $F=L$, the field $F[z]=L[z]$ is a conjugate of $L[x]=L$ over $L$; then $L[z]$ is a conjugation of $L$ over $K$. Hence,
 we have $z\in L[z]\subseteq L$ by the assumption.

Now suppose $F\subsetneqq L$. By \emph{Lemma 2.1} there is a field $M$ of the form
\begin{equation*}
M \triangleq F\left( z,v_{1},v_{2},\cdots ,v_{s},w_{s+1},\cdots, w_{m}\right),
\end{equation*}
that is isomorphic to $L$ over $F$.  As $K \subseteq F$, it is seen that $M$ is a conjugation of $L$ over $K$. From the assumption again,  we have
\begin{equation*}
z\in M\subseteq L.
\end{equation*}
This proves $z \in L$.

$3rd\implies 2ed$. Fixed a conjugation $H$ of $L$ over $K$. Let $x_{0}\in H$.
Take a $(r,n)$-nice basis $w_{1},w_{2},\cdots ,w_{n} $ of $L$ over $K$ and a $K_{0}$-isomorphism $\sigma:H\rightarrow L$
such that $H$ is a $K$-conjugation of $L$ by $\sigma$, where $$
K_{0}\triangleq K(w_{1},w_{2},\cdots ,w_{r} ).
$$

It is seen that that $x_{0}$ must be algebraic over $K_{0}$ since $w_{1}, w_{2},\cdots ,w_{r} $ are all contained in the
intersection of $H$ and $L$.

Then the field $K_{0}[x_{0}]$, which is a conjugate of  $K_{0}[\sigma(x_{0})]$
over $K_{0} $, is a conjugation of $K_{0}[\sigma(x_{0})]$ over $K$.
From the assumption we have $x_{0}\in
K_{0}[x_{0}]\subseteq L$; hence, $H \subseteq L$.
\end{proof}

\begin{proposition}
The following statements are equivalent:
\begin{itemize}
\item $L$ is quasi-galois over $K$.

\item There exists exists one and only one conjugation of $L$  with respect to some nice basis of $L$ over $K$.
\end{itemize}
\end{proposition}

\begin{proof}
Prove $\Longleftarrow$. It is immediate from \emph{Lemma 3.1}.

Prove $\Longrightarrow$. Let $L$ be quasi-galois over $K$ with respect to a $(s,m)$-nice basis $v_{1},v_{2},\cdots ,v_{m}$ of $L$ over
$K$. Take a conjugation $H$ of $L$ over $K$ and an $F$-isomorphism $\tau$ of $H$ onto $L$ such that via $\tau$ $H$ is a conjugate of $L$ over $F$, where
\begin{equation*}
F\triangleq k(v_{1},v_{2},\cdots ,v_{s}).
\end{equation*}
Then we have $F\subseteq H\subseteq L $ from \emph{Lemma 3.1} again. We prove $H=L$.

Hypothesize $H\subsetneqq L$. Take any $x_{0}\in L\setminus H$. Then
 $x_{0}$ is algebraic over $F$ and $H$, respectively. We have
\begin{equation*}
[H:F]=[L:F]< \infty
\end{equation*}
since $H$ is a conjugate of $L$ over $F$. On the other hand, we have
\begin{equation*}
2+[H:F]\leq[{{{{{H[x_{0}]:F}}}}}]\leq[L:F]
\end{equation*}
from $x_{0}\in L\setminus H$, which will be in contradiction.
\end{proof}

Replacing a nice basis by every linear disjoint basis, we have the following lemma and proposition for strong conjugation and strong quasi-galois.

\begin{lemma}
The following statements are equivalent:
\begin{itemize}
\item $L$ is strong quasi-galois over $K$.

\item Each strong conjugation of $L$ over $K$ is contained in $L$.

\item Fixed each  $(\Delta, A)\in \Sigma[L/K]_{ld}$. Then $L$
contains every conjugate of $F\left( x\right) $ over $F$ for any $x\in L\setminus K(\Delta)$ and any subfield $K(\Delta)\subseteq F\subseteq L$.
\end{itemize}
\end{lemma}

\begin{proposition}
The following statements are equivalent:
\begin{itemize}
\item $L$ is strong quasi-galois over $K$.

\item There exists one and only one strong conjugation of $L$ over $K$.
\end{itemize}
\end{proposition}

\begin{remark}
For strong Galois, we have the following results.
\begin{itemize}
\item Let $L$ be a finitely generated extension of a field $K$ such that  $L$ is strong Galois
over $K$. Then $L$ is a separably generated and strong quasi-galois extension over $K$.

\item Conversely, let $L^{\prime}$ be an extension of a field $K^{\prime}$ such that $L^{\prime}$ is strong quasi-galois and separably generated over $K^{\prime}$. Then $L^{\prime}$ is strong Galois
over $K^{\prime}$.

\item From \emph{Proposition 3.9} it is seen that a strong Galois extension is in deed a purely transcendental extension of an algebraic Galois extension of a field with respect to all possible linearly disjoint bases.
\end{itemize}
\end{remark}

\subsection{Strong Galois = a unique strong conjugation + separably generated}

As an analogue of algebraic extensions of fields, we have the following result for a strong Galois extension.

\begin{proposition}
Let $L$ be a finitely generated extension of a field $K$. The following statements are equivalent.
\begin{itemize}
\item $L$ is strong Galois over $K$.

\item $L$ is strong quasi-galois and separably generated over $K$.

\item $L/K$ is separably generated and $L$ has a unique  strong conjugation over $K$.
\end{itemize}
\end{proposition}

\begin{proof}
$1st \Longleftrightarrow 2ed$. Trivial.

$1st \Longleftrightarrow 3rd$. It is immediate from \emph{Proposition 3.4} and \emph{Remark 3.5}.
\end{proof}

Likewise, we also have the same result for absolute Galois.

\begin{remark} Let $L$ be an extension of a field $K$. Then $L$ is absolute Galois over $K$ if and only if $L$ is absolute quasi-galois and separably generated over $K$, i.e.,
$$\emph{absolute Galois} =\emph{a unique absolute conjugation} + \emph{separably generated}.$$
\end{remark}

However, there does not exist such a relation for Galois.

\begin{remark} Let $L$ be an extension of $K$. Then
$L$ is tame Galois over $K$ if  $L$ is quasi-galois and separably generated over $K$. In general, the converse is not true , i.e.,
$$\emph{tame Galois} \not=\emph{a unique conjugation} + \emph{separably generated}.$$
See \emph{Example 1.5} for a counterexample.
\end{remark}

\subsection{A purely transcendental extension is strong Galois}

Now consider a purely transcendental extension of a field.

\begin{proposition}
Let $\Delta$ be a finite or an infinite (countable) set of algebraically independent variables over a field $K$. Then
\begin{itemize}
\item $K(\Delta)$ is Galois over $K$;

\item $K(\Delta)$ is strong Galois over $K$.
\end{itemize}
\end{proposition}

\begin{proof}
$(i)$. Prove $K(\Delta)$ is Galois over $K$.
It suffices to prove that $K$ is the invariant subfield of the Galois group $Gal(K(\Delta)/K)$.

In deed, let $\tau$ be an automorphism of $K(\Delta)$ over $K$ defined by
\begin{equation*}
v\mapsto \tau(v)=\frac{1}{v}
\end{equation*}
for any $v\in \Delta$.

Hence, $\tau \in Gal(K(\Delta)/K)$ is given by
\begin{equation*}
\frac{f(v_{1},v_{2},\cdots ,v_{n})}{g(v_{1},v_{2},\cdots ,v_{n})}\in k(\Delta)
\end{equation*}
\begin{equation*}
\mapsto \frac{f(\tau (v_{1}),\tau (v_{2}),\cdots
, \tau (v_{n}))}{
g(\tau (v_{1}),\tau (v_{2}),\cdots
,\tau (v_{n}))}\in k(\Delta)
\end{equation*}
for any $$v_{1},v_{2},\cdots
,v_{n}\in \Delta$$ and
for any polynomials $$f(X_{1},X_{2},\cdots, X_{n})$$ and $$g(X_{1},X_{2},
\cdots, X_{n})\not= 0$$ over the field $K$ with $$g(v_{1},v_{2},\cdots
,v_{n})\not= 0.$$

It is seen that $\tau$ is well-defined since we have
\begin{equation*}
g(v_{1},v_{2},\cdots ,v_{n})= 0
\end{equation*}
if and only if
\begin{equation*}
g(\tau (v_{1}),\tau (v_{2}),\cdots
,  \tau (v_{n}))=0.
\end{equation*}

Then $K$ is the invariant subfield of the automorphism $\tau $ of $K(\Delta)$  over $K$, from which it follows that $K$ is the invariant subfield of the
Galois group $Gal(K(\Delta)/K)$.

This proves that $K(\Delta)$ is Galois over $K$.

$(ii)$.
Let $\Delta =\{ t_{1},t_{2},\cdots, t_{n},\cdots \}$ be a set of algebraically independent variables over $K$. Prove that $K(\Delta)$ is strong Galois over $K$.

By induction on $n$, it reduces to consider the case that $\Delta=\{t\}$ since we have $$K(u,v)=(K(u))(v)$$ for any $u,v\in \Delta$.

In the following we prove $L \triangleq K(t)$ is strong Galois over $K$.

From definition it is immediate that $L$ is a unique conjugation of $L$ over $K$ with respect to the nice basis $t$.

Take any given linearly disjoint basis $(S,A)$ of $L$ over $K$. We have $L=K(S)[A]$.

It is seen that $A$ and $S$ are both finite sets. In deed, from definition it is immediate that $S\subseteq L$ is a set of algebraically independent variables over $K$ and $A\subseteq L$ is a set of generators of $L$ over the $K(S)$ such that $K(S)[A]$ is algebraic over $K(S)$. Then $S=\{s\}$ has a single element since the dimension of $L$ over $K$ is one. On the other hand, as $t\in K(s)[A]$, there are a finite number of elements $$\omega_{1},\omega_{2},\cdots,\omega_{n}\in A$$ and $$a_{1},a_{2},\cdots,a_{n}\in K(s),$$ respectively, such that
$$t=a_{1}\omega_{1}+a_{2}\omega_{2}+\cdots+a_{n}\omega_{n}.$$
Hence, $A$ must be a finite set since
 $L=K(s)[A]$ and $A$ is a linearly independent set over $K(s)$.

Suppose $$[L:K(s)]=m.$$ As $L=K(s)[A]$, we have $$\omega_{1},\omega_{2},\cdots,\omega_{m-1}\in L$$ such that $$A=\{1,\omega_{1},\omega_{2},\cdots,\omega_{m-1}\}.$$

It is clear that $t$ is  algebraic over $K(s)$. Let $$f(X)=f_{h}X^{h}+f_{h-1}X^{h-1}+\cdots+f_{1}X+f_{0}$$ be the irreducible polynomial of $t$ over $K(s)$ with $f_{h}\not=0$. Then we must have $$h=1, A=\{1\}.$$
In deed, hypothesize $h>1$. Let $\beta\in K(t)$ be a root of $f(X)=0$. It is easily seen that $\beta$ is a linear combination of the elements contained in $A$ with coefficients in $K(s)$. Without loss of generality, suppose $\omega_{1}=\beta$. Then $K(s)$ and $K(A)$ is linearly disjoint over $K$, which will be in contradiction.

It follows that $$t=-\frac{f_{0}}{f_{0}}=\frac{a(s)}{b(s)}$$ where $a(X)$ and $b(X)\not=0$ are polynomials with coefficients in $K$. Then $$L=K(s)=K(t).$$

As $K(s)$ is Galois over $K(s)$ itself, it is seen that
$L=K(t)$ is strong Galois over $K$. This completes the proof.
\end{proof}

As an application of \emph{Proposition 3.9}, we have the following result on a tower of Galois extensions.

\begin{proposition}Let $L$ be a finitely generated purely transcendental extension over a field $K$
and let $M$ be a finitely generated extension over $L$. There are the following statements.
\begin{itemize}
\item Let $M$ be algebraic Galois over $L$. Then $M$ is Galois over $K$. In particular, a tame Galois extension is Galois.

\item Let $M$ be Galois over $L$. Then $M$ is Galois over $K$ if $M$ and $L$ are linearly disjoint over $K$.

\item Suppose that $M$ is strong Galois over $L$. Then $M$ is strong Galois over $K$ if $M$ and $L$ are linearly disjoint over $K$.
\end{itemize}
\end{proposition}

\begin{proof}
Prove the $1st$. Let $M$ be an algebraic Galois extension.
It is seen that $L/K$ is Galois from \emph{Proposition 3.9}. Prove that $M$ is Galois over $K$.

In fact, without loss of generality, assume that we have two automorphisms $\tau_{1}\in Gal(L/K)$ and $\tau_{2}\in Gal(M/L)$
such that $$K=\{x\in L:\tau_{1}(x)=x\};$$ $$L=\{x\in M:\tau_{2}(x)=x\}$$
by induction on the number of generators of the fields since $L/K$ and $M/L$ both are finitely generated.

Then there is an automorphism $\tau \in Gal(M/K)$  induced by $\tau_{1}$
and $\tau_{2}$ in such a manner
\begin{equation*}
\frac{f(v_{1},v_{2},\cdots, v_{m}, \cdots ,v_{n})}{g(v_{1},v_{2},\cdots, v_{m}, \cdots , v_{n})}\in M
\end{equation*}
\begin{equation*}
\mapsto \frac{f(\tau_{1}(v_{1}),\tau_{1}(v_{2}),\cdots
,\tau_{1}(v_{m}),\tau_{2}(v_{m+1}),\cdots, \tau_{2}(v_{n}))}{
g(\tau_{1}(v_{1}),\tau_{1}(v_{2}),\cdots
,\tau_{1}(v_{m}),\tau_{2}(v_{m+1}),\cdots, \tau_{2}(v_{n}))}\in M
\end{equation*}
for any $$v_{1},v_{2},\cdots, v_{m} \in L, \, v_{m+1}, \cdots ,v_{n}\in M$$ and
for any polynomials $$f(X_{1},X_{2},\cdots, X_{n})$$ and $$g(X_{1},X_{2},\cdots, X_{n})\not= 0$$ over the field $K$ with $$g(v_{1},v_{2},\cdots
,v_{n})\not= 0.$$

Hence, we have $$K=\{x\in M:\tau (x)=x\}.$$

It follows that $K$ is the invariant subfield of the Galois group $Gal(M/K)$. This proves that $M$ is Galois over $K$.

Prove the $2ed$. Let $L=K(\Delta_{0})$ be purely transcendental over $K$. Take a nice basis $(\Delta, A)$ of $M$ over $L$ such that $M$ is algebraic Galois over $L(\Delta)$.

As $M$ and $L$ are linearly disjoint over $K$, it is seen that the union $\Delta \cup \Delta_{0}$ is an algebraically independent set over $K$.

It is easily seen that $M=K(\Delta \cup \Delta_{0})[A]$ is an algebraic Galois extension of $K(\Delta \cup \Delta_{0})$. By the first statement, it is seen that $M$ is Galois over $K$.

Prove the $3rd$. Likewise, we give the proof of the third statement.
\end{proof}

\subsection*{Acknowledgment}
The author would like to express his
sincere gratitude to Professor Li Banghe for his advice and instructions on
algebraic geometry and topology.

The author also would like to express his sincere gratitude to Professor
Laurent Moret-Bailly for pointing out an error in an early version of the preprint of the paper.

\end{document}